\documentclass[12pt,reqno]{amsart}
\usepackage{amsmath, mathtools}
\usepackage{amsthm}
\usepackage{amssymb}
\usepackage{enumerate}
\usepackage{amscd}
\usepackage{color}
\usepackage{pb-diagram}
\usepackage{graphicx}
\usepackage[all, cmtip]{xy}
\usepackage{soul}
\usepackage[colorinlistoftodos]{todonotes}
\usepackage{esint}
\usepackage{hyperref}
\hypersetup{pdftex,colorlinks=true,allcolors=blue}
\usepackage{hypcap}
\usepackage[titletoc]{appendix}
\usepackage{comment}
\usepackage{float}
\restylefloat{table}
\usepackage{tikz}
\usetikzlibrary{calc,decorations.pathmorphing,shapes}

\usepackage[colorinlistoftodos]{todonotes}
\usepackage[normalem]{ulem}
\usepackage{soul,xcolor}
\setstcolor{red}
\makeatletter
\def\uwave{\bgroup \markoverwith{\lower3.5\p@\hbox{\sixly \textcolor{red}{\char58}}}\ULon}
\font\sixly=lasy6 
\makeatother

\newenvironment{red}{\relax\color{red}}{\relax}
\newenvironment{blue}{\relax\color{blue}}{\hspace*{.5ex}\relax}

\newcommand{\ber}{\begin{red}}
	\newcommand{\er}{\end{red}}
\newcommand{\beb}{\begin{blue}}
	\newcommand{\eb}{\end{blue}}
\newcounter{sarrow}

\usepackage{amsrefs}

\theoremstyle{plain}
\newtheorem{lemma}{Lemma}
\newtheorem{prop}[lemma]{Proposition}
\newtheorem{coro}[lemma]{Corollary}
\newtheorem{theo}[lemma]{Theorem}
\newtheorem{rema}[lemma]{Remark}
\newtheorem{lemm}[lemma]{Lemma}
\newtheorem{conj}[lemma]{Conjecture}

\theoremstyle{definition}
\newtheorem{defi}[lemma]{Definition}
\newtheorem{thm-Intro}{Theorem} 
\newtheorem{cor-Intro}{Corollary} 

\numberwithin{equation}{section}

\newcommand{\Hol}{\textup{Hol}}


\begin{document}
	\title[Non-measure hyperbolicity of complex K3 surfaces]{Non-measure hyperbolicity of complex K3 surfaces}
	
	\author{Gunhee Cho}
\address{Department of Mathematics\\
	University of California, Santa Barbara\\
South Hall, Room 6607\\
Santa Barbara, CA 93106.}
\email{gunhee.cho@math.ucsb.edu}

	\author{David R. Morrison}
\address{Department of Mathematics\\
	University of California, Santa Barbara\\
South Hall, Room 6607\\
Santa Barbara, CA 93106.}
\email{drm@math.ucsb.edu }
	
	\begin{abstract} 
We show that the non-measure hyperbolicity of K3 surfaces---which M. Green and P. Griffiths verified for certain cases in 1980---holds for  all K3 surfaces. As a byproduct, we prove the non-measure hyperbolicity of any Hilbert schemes of points on K3 surfaces. We also obtain a new proof of the non-measure hyperbolicity of any Enriques surface. 
	\end{abstract}
	
	\maketitle
	
	
	\section{Introduction and results}
	
	The Kobayashi pseudometric on a complex manifold is a generalization of the K\"ahler-Einstein metric with constant negative holomorphic sectional curvature. It corresponds to the Poincar\'e metric in complex hyperbolic space. According to Brody's theorem, the non-existence of entire curves (i.e., non-constant holomorphic maps from $\mathbb{C}^1$ to $M$) on a compact complex manifold $M$ is equivalent to the non-vanishing of the Kobayashi pseudometric \cite{MR0470252}. S. Kobayashi conjectured that all compact Calabi-Yau manifolds have a vanishing Kobayashi pseudometric \cite{MR414940}. Numerous studies have focused on demonstrating the non-existence of entire curves based on Brody's work (for example, \cite{MR0726433,MR3209357,MR3644248,MR4124989, MR3649666}). For a comprehensive introduction to the subject, refer to the cited sources \cite{MR1492539,MR3524135}. Notable problems in this area include the Green-Griffiths and Lang conjectures, which involve the Demailly-Semple vector bundle approach for hypersurfaces of high degree \cite{MR0828820,MR609557}. Another method involves applying deformation to hyper-K\"ahler manifolds, including all complex K3 surfaces, as presented by Kamenova, Lu, and Verbitsky \cite{MR3263959}.
	
A more challenging problem than the vanishing Kobayashi pseudometric is the vanishing Kobayashi-Eisenman pseudovolume. This is because the vanishing Kobayashi-Eisenman pseudovolume implies the vanishing Kobayashi pseudometric. However, the non-existence of entire curves cannot be directly applied, necessitating a new approach. In the case of the Demaily-Semple vector bundle approach, there have been intermittent attempts to extend the domain of entire curves to $\mathbb{C}^k$ instead of $\mathbb{C}^1$, in comparison to the study of entire curves (see, for example, \cite{MR2818709}).

It is also worth noting that the following version of the Kobayashi
conjecture \cite{MR414940} is well-known in the classical literature.
Recent updates on this conjecture can be found in \cite{MR2132645} and
\cite{MR4068832}. This conjecture is based on the observation that
measure hyperbolicity holds for any variety of general type over
$\mathbb{C}$. This implication follows from a result of F. Sakai \cite{MR0590433}, and result of Sakai was generalized to all intermediate Kobayashi--Eisenman measures
in the paper by S. Kaliman \cite{MR1722804} (also see the definition of measure hyperbolicity in
Definition~\ref{def:measure-nonhyperbolicity} below).

	\begin{conj}\label{conjecture:Kobayashi}
		An $n$-dimensional algebraic variety $X$ over $\mathbb{C}$ is measure hyperbolic if and only if $X$ is of general type. 
	\end{conj}
	
Complex K3 surfaces are a significant class of compact complex surfaces
that have been extensively studied in both differential geometry and
algebraic geometry (see, for example, \cite{MR0707352, MR0728142} and
the references therein).  M. Green and P. Griffiths \cite{MR609557}
confirmed Conjecture~\ref{conjecture:Kobayashi} for certain classes of
complex algebraic K3 surfaces $X$,  including 2-sheeted coverings $X \rightarrow
\mathbb{P}^2$ that are branched over a smooth curve of degree six, and
non-degenerate smooth surfaces $X \subset \mathbb{P}^3$ (respectively $X
\subset \mathbb{P}^4$) with degrees 4 (respectively 6) and geometric
genus $p_g \neq 0$. In those three cases, they were able to show that
the complex algebraic K3 surface is rationally swept out by elliptic curves.
It is now known that {\em all} complex algebraic K3 surfaces have this property
\cite[Corollary 13.2.2]{MR3586372}.  It then follows (from
\cite{MR609557} or \cite{MR414940}) that 
Conjecture~\ref{conjecture:Kobayashi} holds for all complex algebraic K3
surfaces. 
	
 In this paper, we verify conjecture~\ref{conjecture:Kobayashi} for all
complex K3 surfaces:

\begin{theo}[Corollary~\ref{cor:vanishing-k3}]\label{thm:main}
Every complex K3 surface $X$ is non-measure hyperbolic.
\end{theo}

The key idea is to deform the complex structure of a complex K3 surface
to achieve one with vanishing Kobayashi-Eisenman pseudovolume. To accomplish
this, we consider the moduli space of (``marked'') 
complex K3 surfaces and focus on
the collection of elliptic K3 surfaces with a section, which forms a
dense subset in this moduli space and satisfies the property of
vanishing Kobayashi-Eisenman pseudovolume. Exploiting the upper
semicontinuity of the Kobayashi-Eisenman pseudovolume (and pseudometric)
under deformations of complex structures, we can establish the vanishing
of volume for any complex K3 surface. Recent research by the first
author has explored how the deformations of complex structure can be
utilized to prove Kobayashi hyperbolic embedding through toric geometry \cite{GCJY21}.

An interesting consequence of Theorem~\ref{thm:main} is the vanishing of the Kobayashi pseudovolume on any $n$th punctual Hilbert schemes of a complex K3 surface, which appears to be previously unknown. Our result provides a stronger condition than the Kobayashi non-hyperbolicity for any complex K3 surface and for any $n$th punctual Hilbert schemes of a complex K3 surface, which had been demonstrated by L. Kamenova, S. Lu, and M. Verbitsky \cite{MR3263959}. 

\begin{coro}[Corollary~\ref{cor:vanishing-Hilbertscheme}]\label{Cor:Hilbert-Scheme}
	Every Hilbert schemes of points on K3 surfaces is non-measure hyperbolic.
\end{coro}

Also, Theorem~\ref{thm:main} yields a new proof of the vanishing of the Kobayashi pseudovolume on Enriques surfaces, which had previously been proven in \cite{MR609557}.

This paper is structured as follows:

Section 2 provides the definitions of the Kobayashi pseudometric and
Kobayashi-Eisenman pseudovolume. In Section 3, we discuss complex K3
surfaces and  their moduli space,  and describe the collection of
elliptic K3 surfaces with a section, which forms a dense subset of this
moduli space. Section 4 focuses on the proof of the vanishing
Kobayashi-Eisenman pseudovolume specifically for elliptic K3 surfaces with a section. Section 5 establishes the vanishing of the Kobayashi-Eisenman pseudovolume on any complex K3 surface, leveraging the upper semicontinuity of the Kobayashi-Eisenman pseudovolume. In conclusion, we examines the vanishing properties of the Kobayashi-Eisenman pseudovolume for complex K3 surfaces and any $n$th punctual Hilbert schemes of a complex K3 surface. Additionally, we address the non-measure hyperbolicity of Enriques surfaces. 

	\subsection*{Acknowledgments}
We thank Ariyan Javanpeykar,  Ljudmila Kamenova, and Steven Lu for helpful correspondence. We also thank Mikhail Zaidenberg for comments on upper-semicontinuity of Kobayashi-Eisenman pseudovolume and relevant references. GC is partially supported by Simons Travel funding. DRM is partially
supported by National Science Foundation Grant  \#PHY-2014226.

\section{Preliminaries}

\subsection{Kobayashi--(Royden) pseudometric, Kobayashi-Eisenman pseudovolume}
We always denote $\triangle$ and $\triangle_r$ the unit disk and the
disk of a radius $r$ in $\mathbb{C}$.  Let $M$ be a complex manifold of
dimension $n$. Let $p \in M$. We denote by $T_{p} M$ (resp.\ $T_M$) the
holomorphic tangent space to $M$ at $p$ (resp.\ the holomorphic tangent bundle).

\begin{defi}
	
	(1) The {\em Kobayashi--Royden pseudo-metric} $KR_{X}$ is defined by
	\begin{equation*}
			KR_{X}(u)=\inf_{\phi}\left\{\frac{1}{|\mu|}, \phi_{*}\left(\frac{\partial}{\partial t}\right)=\mu u\right\}
	\end{equation*}
for $u \in T_{X, x}$. Here $\phi$ runs over the set of holomorphic maps from the unit disk $\triangle$ to $X$, such that $\phi(0)=x$.
	
	(2) The {\em Kobayashi--Eisenman pseudo-volume} $\Psi_{X}$ on the $n$-dimensional complex manifold $X$ is the pseudo-volume form whose associated Hermitian pseudo-norm on $\bigwedge^{n} T_{X}$ is defined by
	\begin{equation*}
		\Psi_{X}(\zeta)=\inf_{\phi}\left\{\frac{1}{|\mu|}, \phi_{*}\left(\frac{\partial}{\partial t_{1}} \wedge \ldots \wedge \frac{\partial}{\partial t_{n}}\right)=\mu \zeta\right\},
	\end{equation*}
	for $\zeta \in \bigwedge^{n} T_{X, x}, n=\operatorname{dim} X$. Here $\phi$ runs over the set of holomorphic maps from $\triangle^{n}$ to $X$ such that $\phi(0)=x$.
	
	(3) More generally the Eisenman p-pseudo-volume form $\Psi_{X}^{p}$ is defined as in (2), replacing $n=\operatorname{dim} X$ by any $p$ comprised between 1 and $n$. It is then defined only on p-vectors $\zeta \in \bigwedge^{p} T_{X, x}$ which are decomposable, that is $\zeta=u_{1} \wedge \ldots \wedge u_{p}$.
	
	We clearly have
	
	$$
	\Psi_{X}^{1}=KR_{X}, \Psi_{X}^{n}=\Psi_{X}.
	$$
\end{defi}

	Intuitively, $KR_{X}(p;v)$ measures the maximal radius of one-dimensional disk embedded holomorphically in $X$ in the direction of $v$ (see also~ \cite{GunheeChoJunqingQian20} for the concrete formula for very simple Riemann surfaces). It is only a pseudo-metric since it is possible that $KR_{X}(p;v)=0$ for some non-zero tangent vector. 
The Kobayashi-Royden pseudometric is well defined as a pseudo-metric on arbitrary complex manifolds, and it coincides with the Poincar\'e metric in the case of a complex hyperbolic space. The Kobayashi pseudo-distance $d_{X}$ on $X$ is defined by \[d_{X}(x,y)=\inf_{f_i \in \Hol(\triangle,X)}\left\{\sum_{i=1}^{n} \rho_{\triangle}(a_i,b_i) \right\}, \] where $x=p_0, ... ,p_n =y, f_i(a_i)=p_{i-1}, f_i(b_i)=p_i$ and $\rho_{\triangle}(a,b)$ denotes the Poincar\'e distance between two points $a,b\in \triangle$. Royden proved that the infimum of the arc-length of all piecewise $C^1$-curves with respect to the Kobayashi--Royden metric induces the Kobayashi pseudo-distance $d_X$ \cite{MR0291494}.

The upper semicontinuity of the Kobayashi-Eisenman measures on a variety 
was established as a lemma in \cite{MR0367300} of Eisenman (publishing
under the surname Pelles). One of the results there says that the $k$th
Kobayashi--Eisenman measure is insensible to removing an analytic subset
of codimension k+1 or more \cite{MR1321580}. For the Kobayshi-Royden
pseudometric, this is a result of \cite{MR0425186}. (Also, see \cite{MR1323721, MR1363172}). 

\begin{defi}\label{def:measure-nonhyperbolicity}
A complex manifold $X$ is said to be {\em measure
	hyperbolic} if the Kobayashi--Eisenman pseudo-volume does not vainish on any open set, or equivalently $\Psi_X$ is nonzero almost everywhere.
\end{defi}

The following decreasing property is clear from the definition of Kobayashi-Eisenman volume. 

\begin{lemm}\label{lemma:decreasing}
	Let $\phi : X \rightarrow Y$ be a holomorphic map between two complex manifolds. Then
	\begin{equation*}
		\phi^{*}\Psi_{Y}\leq \Psi_{X}.
	\end{equation*}
	
\end{lemm}

Also, we have the following: 

\begin{lemm}\label{lemma:product}
	Let $X_1,X_2$ be complex manifolds. Denote the $i$th projection $pr_i : X_1\times X_2 \rightarrow X_i,i=1,2$. Then
	\begin{equation*}
		\Psi_{X_1\times X_2}\leq pr^{*}_1 \Psi_{X_1}\otimes pr^{*}_2 \Psi_{X_2}.
	\end{equation*}

\begin{proof}
 \cite[Lemma 1.12]{MR2132645}.
\end{proof}
\end{lemm}

\section{Complex K3 surfaces}
\begin{defi}
	A complex K3 surface is a compact connected complex manifold $X$
of complex dimension $2$ such that the canonical line bundle is trivial
and $H^1(X,\mathcal{O}_X)=0$.  Note that the triviality of the canonical
bundle is equivalent to the existence of a nowhere-vanishing holomorphic 
$2$-form $\omega_X$, and that such a $2$-form is unique up to a
(complex) scalar multiple. 
\end{defi}

In this paper, `K3 surface' shall mean a complex K3 surface in this sense. Not all K3 surfaces are algebraic. 

\subsubsection{Moduli space of K3 surfaces}

There is a 
well-known and well-studied moduli space for K3 surfaces which
is formulated in terms of the {\em periods}\/ of the holomorphic
$2$-form.
There are several versions of this moduli space; the one of greatest
interest to us will be the moduli space of {\em marked} K3 surfaces.
To construct this, we declare a {\em marking} of a complex K3 surface
$X$ to be a choice of isometry $\alpha: H^2(X,\mathbb{Z}) \to L$,
where $L$ is a fixed free ${\mathbb Z}$-module of rank $22$ carrying
a $\mathbb Z$-valued  inner product $\langle\quad,\quad\rangle$
of signature $(3,19)$ such that $\langle v,v\rangle$
is an even integer for each $v\in L$.  (Such a module is known \cite{MR0255476}
to be isometric
to the orthogonal direct sum of two copies of the sign-reversed $E_8$ lattice
plus three copies of the ``hyperbolic plane'' $U$ (the even integral inner
product space whose inner product has matrix
$\begin{bmatrix}0&1\\1&0\end{bmatrix}$ 
in an appropriate basis)).
A marking induces a dual map (again denoted by $\alpha$):
$\alpha:H_2(X,\mathbb{Z}) \to L^*\cong L$.
Two marked K3 surfaces are said to be isomorphic if there is an
isomorphism compatible with the markings; the set of isomorphism classes
of marked K3 surfaces is known to be a complex manifold of complex
dimension 20, although it is not Hausdorff (see \cite[Section
7.2]{MR3586372}).

The moduli space of marked K3 surfaces can also be regarded as the
{\em Teichm\"uller space}\/ of K3 surfaces.  Since autorphisms of
 a K3 surface $M$ always act nontrivially on its second cohomology
\cite[Proposition 15.2.1]{MR3586372}, the
mapping class group 
$$
       \Gamma:=\operatorname{Diff}(M) / \operatorname{Diff}_{0}(M)
       $$ 
(where $\operatorname{Diff}_{0}(M)$ is the connected component of the
diffeomorphism group) can be regarded as the group which permutes
markings.  Thus, if we let Comp denote the space of complex structures
on $M$, equipped with a structure of a Fr\'echet manifold, and we
let the Teichm\"uller space be
$\operatorname{Teich} := \operatorname{Comp} / \operatorname{Diff}_{0}(M)$
then the quotient Comp / Diff $=$ Teich $/ \Gamma$ is the moduli space
of complex structures on $M$, and $\operatorname{Teich}$ is the moduli
space of marked K3 surfaces.

Given a marked K3 surface $X$ with marking $\alpha$, one fixes
a basis $\gamma_1, \dots,
\gamma_{22}$ of $L^*\cong L$
and for each
nonzero holomorphic $2$-form $\omega_X$ on $X$ 
defines the {\em period point}
$$
\pi(X,\alpha) :=(\int_{\alpha^{-1}(\gamma_1)}\omega_X,\dots,\int_{\alpha^{-1}(\gamma_{22})}\omega_X)\in
\mathbb{P}^{21}\
$$
which is a well-defined point in projective space independent of
$\omega_X$
since $\omega_X$ is unique up to a 
nonzero complex multiple.
Each $\pi(X,\alpha)$ lies in the subset
$$ \Omega:=\{ \omega \in \mathbb{P}^{21} | \langle \omega, \omega \rangle =
0\}$$
where we endow $\mathbb{P}^{21}=\mathbb{P}(L\otimes_{\mathbb{Z}} \mathbb{C})$ with
the inner product $\langle\quad,\quad\rangle$ induced from that
of $L$.  
The {\em period map} $(X,\alpha) \mapsto \pi(X,\alpha)$ is known to be 
a local isomorphism (this is the ``local Torelli theorem''), and this is what gives the moduli space the structure
of a complex manifold of complex dimension 20.

\subsection{Elliptic K3 surfaces with section}

An elliptic K3 surface $X$ is a K3 surface which admits a holomorphic
mapping $\pi:X\to\mathbb{P}^1$ whose general fiber $\pi^{-1}(t)$ is
a complex curve of genus one.  We say that $X$ is an elliptic K3 surface
with a section if there is a curve $C_0\subset X$ (called the
``section'') such that $\pi|_{C_0}$ establishes an isomorphism from $C_0$
to $\mathbb{P}^1$.

It is known that a K3 surface $X$ is an elliptic K3 surface with a section 
if and only if there
exists an embedding $U \rightarrow NS(X)$ of the hyperbolic plane into
the N\'eron-Severi group $NS(X)$ of $X$.  (For a marked K3 surface,
NS(X) can be identified with $\pi(X,\alpha)^\perp \cap
\overline{\pi(X,\alpha)}^\perp\cap L$.)

Elliptic K3 surfaces with section are quite common among K3 surfaces,
as shown by the following proposition
(proven in \cite[Remark 14.3.9]{MR3586372} or \cite[Chapter VIII]{MR749574}). 

\begin{prop}\label{prop:dense_subset}
Elliptic K3 surfaces with a section are dense in the moduli space of all
marked  K3 surfaces and also in the moduli spaces $M_d$ of polarized K3
surfaces of fixed degree. 	
\end{prop}

\section{Vanishing Kobayashi-Einsenman volume on Elliptic k3 surfaces with a section}

In this section, we prove 

\begin{prop}\label{prop:vanishing1}
For any elliptic K3 surface $\pi : X \rightarrow \mathbb{P}^1$ with a section, the Kobayashi-Eisenman volume on $X$ is vanishing. 
\end{prop}

Proposition~\ref{prop:vanishing1} follows from \cite{MR1738063} which
assumes dominability by $\mathbb{C}^2$, but we proceed with the explicit construction. 

For an elliptic K3 surface $\pi : X \rightarrow \mathbb{P}^1$ with a
section, the section meets every fibre transversally, and it meets a singular fibre $X_t=\sum m_i C_i$ in precisely one of the irreducible components $C_{i_0}$ and in a smooth point for which $m_{i_0}=1$. By Kodaira's classfication of singular fibres the reduced curve $\sum_{i\neq i_0}C_i$ is an ADE curve which can be contracted to a simple surface singularity.

Due to the fact that $\pi : X \rightarrow \mathbb{P}^1$ always has a Weierstrass model $\overline{X}$ which is birational to $X$, Proposition~\ref{prop:vanishing1} immediately follows from Proposition~\ref{prop:birational} and Proposition~\ref{prop:non-constant-map}.

\begin{prop}\label{prop:birational}
	Let $\tau : M \rightarrow M'$ be a birational equivalence of complex manifolds. Suppose that the Kobayashi-Eisenman pseudo-volume on $M$ vanishes. Then it vanishes on $M'$. 
	\begin{proof}
		Apply \cite[Lemma 1.14]{MR2132645}. 
	\end{proof}	
\end{prop}

\begin{prop}\label{prop:non-constant-map}
	Consider an elliptic K3 surface $\pi : X \rightarrow \mathbb{P}^1$ with a section. Then the Kobayashi-Eisenman volume of the Weierstrass model $\overline{X}$ vanishes. 
\end{prop}

\begin{proof}[Proof of Proposition~\ref{prop:non-constant-map}]
We first construct the Weierstrass model $\overline{X}$ of $X$. The
construction of Weirestrass model for an elliptic K3 surface with a
section is well-known (for example, \cite[Section 11.2]{MR3586372}). We provide the proof for the sake of completeness. 	
	
Fix a section $C_0 \subset X$ of $\pi : X \rightarrow \mathbb{P}^1$. Then we construct the desired map from a Weierstrass model $\overline{X}$ of $X$ as follows: the exact sequence $0 \rightarrow \mathcal{O}_X \rightarrow \mathcal{O}(C_0) \rightarrow \mathcal{O}_{C_0}(-2) \rightarrow 0$ induces the long exact sequence
\begin{equation*}
	0 \rightarrow \mathcal{O}_{\mathbb{P}^1} \rightarrow \pi_{*}\mathcal{O}(C_0) \rightarrow \mathcal{O}_{\mathbb{P}^1}(-2) \rightarrow R^1\pi_{*}\mathcal{O}_X \rightarrow 0,
\end{equation*}
where the vanishing $R^1\pi_{*}\mathcal{O}(C_0)=0$ follows from the corresponding vanishing on the fibres. Note that $\pi_{*}\mathcal{O}(C_0)$ is a line bundle, as $h^0(X_t, \mathcal{O}(p))=1$ for any point $p\in X_t$ in an arbitrary fibre $X_t$. (It is enough to test smooth fibres, as $\pi_{*}\mathcal{O}(C_0)$ is torsion free.) Thus, the cokernel of $\mathcal{O}_{\mathbb{P}^1} \rightarrow \pi_{*}\mathcal{O}(C_0)$ is torsion, but also contained in the torsion free $\mathcal{O}_{\mathbb{P}^1}(-2)$. Hence, 
\begin{equation*}
	\mathcal{O}_{\mathbb{P}^1}\simeq \pi_{*} \mathcal{O}(C_0) \text{ and } \mathcal{O}_{\mathbb{P}^1}(-2)\simeq R^1 \pi_{*} \mathcal{O}_X.
\end{equation*} 
Similarly, using the short exact sequences
\begin{equation*}
	0 \rightarrow \mathcal{O}(C_0) \rightarrow \mathcal{O}(2C_0) \rightarrow \mathcal{O}_{C_0}(-4) \rightarrow 0,
\end{equation*}
and 
\begin{equation*}
	0 \rightarrow \mathcal{O}(2C_0) \rightarrow \mathcal{O}(3C_0) \rightarrow \mathcal{O}_{C_0}(-6) \rightarrow 0,
\end{equation*}
we can deduce that $\pi_{*} \mathcal{O}(2C_0)\simeq \mathcal{O}_{\mathbb{P}^1}(-4)\oplus \mathcal{O}_{\mathbb{P}^1}$ and $\pi_{*} \mathcal{O}(3C_0)\simeq \mathcal{O}_{\mathbb{P}^1}(-4)\oplus \mathcal{O}_{\mathbb{P}^1}(-6)\oplus \mathcal{O}_{\mathbb{P}^1}$. Let $F:= \mathcal{O}_{\mathbb{P}^1}(-4)\oplus \mathcal{O}_{\mathbb{P}^1}(-6)\oplus \mathcal{O}_{\mathbb{P}^1}$. Thus, the linear system $\mathcal{O}(3C_0)_{|_{X_t}}$ on the fibres $X_t$ (or rather the natural surjection $\pi_{*} F = \pi^{*}\pi_{*} \mathcal{O}(3C_0)\rightarrow \mathcal{O}(3C_0)$) defines a morphism
\begin{equation*}
	\varphi : X \rightarrow \mathbb{P}(F^{*})
\end{equation*}
with $\varphi^{*}\mathcal{O}(1)\simeq \mathcal{O}_p(3C_0)$, which is a
closed embedding of the smooth fibres and contracts all components of
singular fibres $X_t$ that are not met by $C_0$ (since
$\mathcal{O}(3C_0)$ is indeed base point free on all fibres). The image
$\overline{X}$ is the {\em Weierstrass model} of the elliptic surface
$X$.
Using the Riemann-Roch theorem on the fibers, one learns that there is
an equation satisfied by the image $\overline{X}\subset \mathbb{P}(F^{*})$, which can be regarded as a section
\begin{equation*}
        f \in H^{0}(\mathbb{P}(F^{*}),\mathcal{O}_p(3)\otimes
p^{*}\mathcal{O}_{\mathbb{P}^1}(6d)).
\end{equation*}
for some degree $d$. 

In order to determine which $d$ corresponds to a K3 surface, we use the
adjunction formula 
$\omega_{\overline{X}}\simeq (\omega_{\mathbb{P}(F^{*})}\otimes
\mathcal{O}(\overline{X}))|_{\overline{X}}$ and the relative Euler
sequence expressing $\omega_{\mathbb{P}(F^{*})}$ to show that
$\omega_{\overline{X}}=O_{\overline{X}}$ if and only if
$\mathcal{O}(\overline{X})\simeq \mathcal{O}_p(3)\otimes
p^{*}\mathcal{O}_{\mathbb{P}^1}(12)$, i.e., $d=2$. 

Now use $H^{0}(\mathbb{P}(F^{*}),\mathcal{O}_p(3)\otimes p^{*}\mathcal{O}_{\mathbb{P}^1}(12))\simeq H^{0}(\mathbb{P}^1, Sym^3(F)\otimes \mathcal{O}_{\mathbb{P}^1}(12))$ and view $x,y$, and $z$ as the local coordinates of the direct summands $\mathcal{O}_{\mathbb{P}(-4)},\mathcal{O}_{\mathbb{P}(-6)}$, and $\mathcal{O}_{\mathbb{P}}$ of $F$. 
By a change of coordinates, the equation $f$ can be put into Weierstrass
form (cf.~\cite{silverman}):
\begin{equation}\label{eq:Weierstrass}
y^2z=4x^3-g_2 xz^2-g_3 z^3
\end{equation}
with coefficients
\begin{equation*}
	g_2 \in H^0(\mathbb{P}^1,\mathcal{O}_{\mathbb{P}^1 }(8)), \text{
} g_3 \in H^0(\mathbb{P}^1,\mathcal{O}_{\mathbb{P}^1 }(12)),
\end{equation*}
The terms in the Weierstrass equation can be seen as a sections of $p_{*}\mathcal{O}_p(3)\otimes
\mathcal{O}_{\mathbb{P}^1}(12)$, which implies, for example, that
$g_2$ can be interpreted as a
section of
$\mathcal{O}_{\mathbb{P}^1}(8)=[xz^2]\mathcal{O}(-4)_{\mathbb{P}^1}\otimes
\mathcal{O}_{\mathbb{P}^1}(12)$. The discriminant is the non-trivial
section $\triangle:={g_2}^3-27{g_3}^2\in H^0(\mathbb{P}^1,\mathcal{O}_{\mathbb{P}^1}(24))$.
Applying the standard coordinate changes, one can always reduce to the
situation that $f$ has this form and $(g_2,g_3)$ is unique up to passing
to $(\lambda^4 g_2, \lambda^6 g_3)$, where $\lambda$ is a function which
is non-vanishing away from the singular fibers.

For a given elliptic K3 surface with section $X$, we let $g_2(t)$ and
$g_3(t)$ be Weierstrass coeffients, and
$\triangle(t)=g_2(t)^3-27g_3(t)^2$ be the discriminant locus.  Let $S_X$ be the set of zeroes of the discriminant $\triangle(t)$ for 
$t\in\mathbb{P}^1$, and $R_X=\mathbb{P}^1 - S_X$.
We define a holomorphic map $F : \mathbb{C}\times R_X \rightarrow
\overline{X}$ by using the Weierstrass $\wp$-function associated with
~\eqref{eq:Weierstrass}, as follows (see \cite{MR1027834}). First, to a pair $(g_2,g_3)$ of
complex numbers satisfying $g_2^3-27g_3^2\ne0$, we associate a {\em lattice} $\Lambda\subset\mathbb{C}$
which is the set of  possible values
$$\omega = \int_\gamma \frac{dx}{\sqrt{4x^3-g_2x-g_3}} ,$$
where $\gamma$   ranges over all possible contours of
integration defined on the double cover of the complex $x$-plane defined
by $ \sqrt{4x^3-g_2x-g_3}$.  (The condition $g_2^3-27g_3^2\ne0$ ensures
that the cubic $4x^3-g_2x-g_3$ does not have any repeated roots, so that
the
double cover is ramified at precisely $4$ points, including infinity.)  Note
that the elements $\omega\in \Lambda$ depend holomorphically on $g_2$
and $g_3$.

The Weierstrass $\wp$-function is defined by
$$ \wp_{\Lambda}(z) = \frac1z + \sum_{0\ne\omega\in\Lambda}
\left\{\frac1{(z-\omega)^2} - \frac1{\omega^2}\right\},$$
which converges uniformly on compact subsets of $\mathbb{C}-\Lambda$ to
a doubly-periodic meromorphic function with poles at $z=\omega\in
\Lambda$ and periods $\wp_{\Lambda}(z+\omega) = \wp_{\Lambda}(z)$ 
for $\omega\in\Lambda$.  Note that this function also depends
holomorphically on $g_2$ and $g_3$.

Using the Weiestrass $\wp$-function, we define a holomorphic map $F :
\mathbb{C}\times R_X \rightarrow
\overline{X}$ starting from the meromorphic map
$F=[\wp_\Lambda(z),\wp'_\Lambda(z),1]$ for $z\in\mathbb{C}$ and
extending holomorphically over the poles of $\wp_\Lambda$.
The image of $F$ satisfies 
\eqref{eq:Weierstrass} 
with Weierstrass coefficients
$$ g_2(\Lambda) = 60\sum_{0\ne\omega\in\Lambda}\frac1{\omega^4}; \quad
g_3(\Lambda) = 140\sum_{0\ne\omega\in\Lambda}\frac1{\omega^6}$$
(see for example \cite{MR1027834}).

Now apply Lemma~\ref{lemma:decreasing}, i.e., the decreasing property of the Kobayashi-Eisenman volume. Since the Kobayashi-Eisenman (in fact Kobayashi-Royden metric) pseudovolume on the first component $\mathbb{C}$ is zero and by Lemma~\ref{lemma:product}, the proof is complete.

\end{proof}

\section{Vanishing Kobayashi-Eisenman pseudo-volume on K3 surfaces}

\subsection{Upper semi-continuity of Kobayashi-Eisenman volume}
We are interested in the upper semicontinuity of $\Psi_{X_t}$ in the variable $t$ for a proper smooth fibration $\pi : \mathcal{M} \rightarrow T$, i.e., $\pi$ is holomorphic, surjective, having everywhere of maximal rank and connected fibers $X_t=\pi^{-1}(t)$. We say a function $F$ on a topological space $X$ with values $\mathbb{R}\cup \left\{ \infty\right\}$ is upper semi-continuous if and only if $\left\{x\in X : F(x)<\alpha \right\}$ is an open set for every $\alpha \in \mathbb{R}$. It is upper semi-continuous at a point $x_0\in X$ if for all $\epsilon>0$ there is a neighbourhood of $x_0$ containing $\left\{x \in X : F(x) < F(x_0)+\epsilon \right\}$. If $X$ is a metric space, this is equivalent to 
\begin{eqnarray}
	\limsup_{t_i \rightarrow t_0}F(t_i) \leq F(t_0),
\end{eqnarray}
for all sequence $(t_i)$ converging to $t_0$.

We will be interested in the upper semicontinuity of $\Psi_{M_{t}}$ in the variable $t$ for a proper smooth fibration $\pi: \mathcal{M} \rightarrow \Delta$, i.e., $\pi$ is holomorphic, surjective, having everywhere of maximal rank and connected fibers $M_{t}=\pi^{-1}(t)$ by the local Torelli Theorem. We apply the following result of M. Zaidenberg. 

\begin{prop}\label{prop:upper-semicontinuity}\cite[Theorem 4.4]{MR791317}
$f: M \rightarrow \Delta$ is a surjective holomorphic mapping with smooth fibers $D_{c}=f^{-1}(c)(c \in \Delta)$. We fix an arbitrary tubular neighborhood $U$ of the fiber $D_{0}$ and a smooth retraction $\pi: U \rightarrow D_{0}$. Fix an Hermitian metric $h$ on $D_{0}$ and denote the associated volume form by $\omega$. We set $\pi_{c}=\pi \mid U \cap D_{c}$. If the Kobayashi-Eisenman pseudo volume form $\mathrm{\Psi}_{D_{0}}$ is continuous, then for each domain $V_{0} \Subset D_{0}$ and every $\varepsilon>0$ there is a $\delta_{0}>0$ such that for all $c \in \Delta_{\delta_{0}}^{*}$ the following inequality holds:
$$
{\Psi_{D_{c}}}  \leq \pi_{c}^{*}\left(\left[(1+\varepsilon){\Psi_{D_{0}}}+\varepsilon|\omega|\right] \right). 
$$

\end{prop}

For the moduli space of marked K3 surfaces, since in our case the central fiber is an elliptic K3 surface with a section, the pseudovolume is the constant $0$. Thus, the continuity at the central fiber implies having the upper semicontinuity with respect to $t$.

\begin{coro}\label{cor:upper-semicontinuity-volume}
	For $M$ a compact complex manifold, let $\operatorname{Vol}(M):=\int_{M} \Psi_{M}\wedge \overline{\Psi_M}$ be the volume of $M$ with respect to $\Psi_{M}$. Then $\operatorname{Vol}(M)$ is upper semicontinuous with respect to the variation of the complex structure on $M$.
\end{coro}

The proof is based on an adjustment of an argument of L. Kamenova, S. Lu, and M. Verbitsky \cite{MR3263959} to our setting. 

\begin{proof}
	We need to show that $\operatorname{Vol}\left(M_{t}\right)$ is upper semicontinuous with respect to $t$ for a family as given above, i.e. for all $t_{0} \in T$ and sequences $\left(t_{i}\right)$ converging to $t_{0}$,
	$$
	\limsup _{t_{i} \rightarrow t_{0}} \operatorname{Vol}\left(M_{t_{i}}\right) \leqslant \operatorname{Vol}\left(M_{t_{0}}\right).
	$$
	If the inequality is false, then after replacing the sequence $\left(t_{i}\right)$ by a subsequence there is an $\varepsilon>0$ such that $\operatorname{Vol}\left(M_{t_{i}}\right)>\operatorname{Vol}\left(M_{t_{0}}\right)+\varepsilon$ for all $i$. Then by Proposition~\ref{prop:upper-semicontinuity},
	$$
	\operatorname{Vol}\left(M_{t_{0}}\right)+\varepsilon \leqslant \limsup _{i \rightarrow \infty}\operatorname{Vol}\left(M_{t_i}\right)=\limsup _{i \rightarrow \infty}\int_{M_{t_i}} \Psi_{M_{t_i}}\wedge \overline{\Psi_{M_{t_i}}}  \leqslant \int_{M_{t_0}} \Psi_{M_{t_0}}\wedge \overline{\Psi_{M_{t_0}}}=\operatorname{Vol}\left(M_{t_{0}}\right),
	$$
	which is a contradiction.
\end{proof}

\begin{rema}
	In general, the Kobayashi-Eisenman pseudovolume does not
satisfy the lower-semicontinuity under deformation of complex
structures. i.e., the ``jumping phenomenon" under deformation. For such an example, \cite[Proposition 9.7]{MR0776396}.
\end{rema}

\subsection{Teichm\"uller spaces and Ergodicity}
We summarize the definition of the Teichmüller space of hyperk\"ahler manifolds, following \cite{MR3413979}:

\begin{defi}\cite[Definition 1.4, 1.6]{MR3413979}
	Let $M$ be a compact complex manifold and
$\operatorname{Diff}_{0}(M)$ a connected component of its diffeomorphism
group (the group of isotopies). Denote by Comp the space of complex
structures on $M$, equipped with a structure of Fr\'echet manifold. We let Teich $:=\operatorname{Comp} / \operatorname{Diff}_{0}(M)$ and call it the Teichmüller space of $M$.  Let $\operatorname{Diff}(M)$ be the group of orientable diffeomorphisms of a complex manifold $M$. Consider the mapping class group
	$$
	\Gamma:=\operatorname{Diff}(M) / \operatorname{Diff}_{0}(M)
	$$	
	acting on Teich. The quotient Comp / Diff $=$ Teich $/ \Gamma$ is called the moduli space of complex structures on $M$. The set Comp / Diff corresponds bijectively to the set of isomorphism classes of complex structures.
\end{defi}

\begin{defi}\cite[Definition 1.17]{MR3413979}
	Let $M$ be a complex manifold, Teich its Teichmüller space, and $I \in$ Teich a point. Consider the set $Z_{I} \subset$ Teich of all $I^{\prime} \in$ Teich such that $(M, I)$ is biholomorphic to $\left(M, I^{\prime}\right)$. Clearly, $Z_{I}=\Gamma \cdot I$ is the orbit of I. A complex structure is called ergodic if the corresponding orbit $Z_{I}$ is dense in Teich.
\end{defi}

For the proof of Corollary~\ref{cor:vanishing-k3}, we use the following theorem: 

\begin{theo}\cite[Theorem 1.16, 4.11]{MR3413979}\label{theorem:ergodicity}
	Let $M$ be a maximal holonomy hyperk\"ahler manifold (which
includes K3 surfaces) or a
compact complex torus of dimension $\geq 2$, and $I$ a complex structure
on $M$. Then $I$ is non-ergodic iff the Neron-Severi lattice of $(M,I)$
has maximal rank (i.e., the Picard number of $(M,I)$ is maximal). 
\end{theo}

\begin{prop}\label{prop:vanishing-ergodic}
	Let $M$ be a complex manifold with vanishing Kobayashi-Eisenman pseudovolume. Then the volume of the Kobayashi-Eisenman pseudovolume $\Psi_M$ defined in Corollary~\ref{cor:upper-semicontinuity-volume} vanishes for all ergodic complex structures in the same deformation class.
\end{prop}

\begin{proof}
	Let Vol : Teich $\longrightarrow \mathbb{R}^{\geqslant 0}$ map a complex structure $I$ to the volume of the Kobayashi-Eisenman pseudovolume on $(M, I)$. By Corollary~\ref{cor:upper-semicontinuity-volume}, this function is upper semi-continuous. Let $I$ be an ergodic complex structure. The set of points $I^{\prime} \in$ Teich such that $\left(M, I^{\prime}\right)$ is biholomorphic to $(M, I)$ is dense, because $I$ is ergodic. By upper semi-continuity, $0=\operatorname{Vol}(I) \geqslant$ $\sup _{I^{\prime} \in \text{Teich}} \operatorname{Vol}\left(I^{\prime}\right)$.
\end{proof}

For the following Corollary, we adapt the similar argument of \cite[Corollary 2.2]{MR3263959}.

\begin{coro}\label{cor:vanishing-k3}
	Let $X$ be a complex K3 surface. Then the Kobayashi-Eisenman pseudovolume on $X$ vanishes almost everywhere (with respect to the associated volume form of any hermitian metric on $X$). 
\end{coro}
\begin{proof}
For any elliptic K3 surface $X'$ with a section, $\Psi_{X'}\equiv 0$ follows from Proposition~\ref{prop:non-constant-map} and the decreasing property of the Kobayashi-Eisenman pseudovolume. 	
	
For general K3 surfaces $X$, we use the following well-known fact: when the Picard number of a K3 surface is greater or equal to $5$, then such a K3 surface admits an elliptic fibration, and when the Picard number is greater or equal to $12$, then such a fibration exists having a
section (see for instance \cite[Chapter 11.1]{MR3586372}). So, if the K3 surface $(X,I)$ does not have an elliptic fibration with section, then
at least the Picard number cannot be $20$ (the maximal value). By Theorem~\ref{theorem:ergodicity}, $I$ is an ergodic complex structure.
Since any two complex structures of K3 surfaces are deformation equivalent to each other, we can find an ergodic complex structure which is an elliptic K3 surface with a section. Therefore  for general K3 surfaces, Proposition~\ref{prop:vanishing-ergodic} implies $\operatorname{Vol}=\int_X \Psi_X \wedge \overline{\Psi_X} =0$, and thus $\Psi_X=0$ almost everywhere.
\end{proof}

One another vanishing consequence is on the Hilbert schemes of points on any complex K3 surface. Let $S$ be a smooth quasi-projective algebraic variety over a field $k$, denotes the symmetric group $\Sigma_{n}$, and symmetric powers $\operatorname{Sym}^{n}(S):=S^{n} / \Sigma_{n}$, which is the moduli space of effective $0$-cycles on $S$ that we only record the multiplicity when points come together. When $X=S$ is a surface, it is well-known that the Hilbert scheme $Hilb^{n}(S)$, which is the moduli of 0-dimensional subschemes of length $n$ in $S$, is a smooth, irreducible variety which is a resolution of singularities of $\operatorname{Sym}^{n}(S)$.

\begin{coro}\label{cor:vanishing-Hilbertscheme}

	Let $Hilb^{n}(X)$ be a Hilbert schemes of points on any complex K3 surface $X$ for any $n$. Then the Kobayashi-Eisenman pseudovolume on $Hilb^{n}(X)$ vanishes almost everywhere.
\end{coro}
\begin{proof}
	For any elliptic K3 surface $X'$ with a section, from the proof of Proposition~\ref{prop:non-constant-map}, we have a non-constant holomorphic function $F : 	\mathbb{C}\times R_{X'} \rightarrow \overline{X'}$. Consequently, from the fact that $Hilb^{n}(X')$ is a resolution of singularities of $\operatorname{Sym}^{n}(X')$, we have a non-constant holomorphic map $f : \mathbb{C}\times Z \rightarrow Hilb^{n}(X')$, where $Z$ is a complex manifold of dimension one less than the dimension of $Hilb^{n}(X')$. By Lemma~\ref{lemma:decreasing},  the Kobayashi-Eisenman pseudovolume vanishes almost everywhere.
	
	For general K3 surfaces $X$, since $Hilb^{n}(X)$ is deformation equivalent to $Hilb^{n}(X')$ with some elliptic K3 surface $X'$ with a section, as the proof of Corollary~\ref{cor:vanishing-k3}, the result follows. 
\end{proof}

The following Corollary concerning Enriques surfaces is also proven in \cite{MR609557}. 

\begin{coro}
	Let $Y$ be a Enriques surface. Then the Kobayashi-Eisenman pseudovolume on $Y$ vanishes. 
\begin{proof}
Every Enriques surface $Y$ admits a complex K3 surface $X$ as a
holomorphic double covering $\phi : X \rightarrow Y$. 
The conclusion follows from Lemma~\ref{lemma:decreasing} and Corollary~\ref{cor:vanishing-k3}. 
\end{proof}

\end{coro}

	\bibliography{reference}
	
\end{document}